\documentclass[12pt]{article}
\usepackage{url}
\usepackage{algorithm,algorithmic}
\usepackage[a4paper]{geometry}        % DIN A4 paper
\RequirePackage[isolatin]{inputenc}
\RequirePackage{comment}
\RequirePackage{graphicx}
\RequirePackage{amsmath,amsfonts}

\ifx \skiptheorems \thiscommandisnotdefinedtralala
\RequirePackage{amsthm}

\ifx \gillesfrench \thiscommandisnotdefinedtralala

\newtheorem{theorem}{Theorem}[section]
\newtheorem{proposition}[theorem]{Proposition}
\newtheorem{corollary}[theorem]{Corollary}
\newtheorem{lemma}[theorem]{Lemma}

\theoremstyle{definition}
\newtheorem{definition}[theorem]{Definition}
\newtheorem{property}[theorem]{Property}
\newtheorem{example}[theorem]{Example}

\theoremstyle{remark}
\newtheorem{remark}[theorem]{Remark}

\else

\newtheorem{theorem}{Théorème}[section]
\newtheorem{proposition}[theorem]{Proposition}

\newtheorem{lemma}[theorem]{Lemme}

\theoremstyle{definition}

\theoremstyle{remark}

\fi
\fi

\newcommand{\mtc}{\mathcal}
\newcommand{\mbf}{\mathbf}

\newcommand{\wt}[1]{{\widetilde{#1}}}
\newcommand{\wh}[1]{{\widehat{#1}}}

\newcommand{\paren}[1]{\left(#1\right)}
\newcommand{\brac}[1]{\left[#1\right]}
\newcommand{\inner}[1]{\left\langle#1\right\rangle}
\newcommand{\norm}[1]{\left\|#1\right\|}
\newcommand{\set}[1]{\left\{#1\right\}}
\newcommand{\abs}[1]{\left\lvert #1 \right\rvert}

\newcommand{\e}[1]{\mbe\brac{#1}}
\newcommand{\ee}[2]{\mbe_{#1}\brac{#2}}
\newcommand{\prob}[1]{\mbp\brac{#1}}

\def\argmin{\mathop{\rm Arg\,Min}\limits}

\newcommand{\eps}{\varepsilon}

\def\cH{{\mtc{H}}}

\def\cK{{\mtc{K}}}
\def\cL{{\mtc{L}}}

\def\cN{{\mtc{N}}}

\def\cP{{\mtc{P}}}

\def\cX{\mathcal{X}}

\newcommand{\mbe}{\mathbb{E}}
\newcommand{\mbr}{\mathbb{R}}
\newcommand{\mbp}{\mathbb{P}}

\newcommand{\bY}{{\mbf{Y}}}

\newcommand{\Y}{\mbf{Y}}
\newcommand{\tr}{\mathrm{tr}}

\title{Optimal learning rates for\\Kernel Conjugate Gradient regression}

\author{
Gilles Blanchard \quad \quad Nicole Kr\"amer\\Weierstrass Institute for\\Applied Analysis and Stochastics\\
Mohrenstr. 39, 10117 Berlin, Germany\\ \texttt{$\{$gilles.blanchard;nicole.kraemer$\}$@wias-berlin.de}
}

\newcommand{\blambda}{\wt{\lambda}}

\begin{document}

\maketitle
%\maketitle
\begin{abstract}
  We prove rates of convergence in the statistical sense for kernel-based
least squares regression using a conjugate gradient algorithm,
where regularization against overfitting is obtained by early stopping.
This method is directly related to Kernel Partial Least Squares, a
 regression method that combines supervised dimensionality reduction with least squares projection.
The rates depend on two key quantities: first, on the regularity of the target
regression function and second, on the intrinsic dimensionality of the data
mapped into the kernel space. Lower bounds
on attainable rates depending on these two quantities were established in
earlier literature, and we obtain upper bounds for the considered method
that match these lower bounds (up to a log factor) if  the
true regression function belongs to the reproducing kernel Hilbert space. If
this assumption is not fulfilled, we obtain similar convergence rates
provided additional unlabeled data are available. The order of the learning rates match state-of-the-art results that were recently obtained for least squares
support vector machines and for linear regularization operators.
\end{abstract}

\section{Introduction}
The contribution of this paper is the learning theoretical analysis of kernel-based least squares regression in combination with conjugate gradient techniques. The goal is to estimate  a regression function $f^*$ based
on random noisy observations. We
have an i.i.d.~sample of $n$ observations $(X_i,Y_i)\in \mathcal{X} \times
\mbr$ from an unknown distribution $P(X,Y)$ that follows the model
\begin{eqnarray*}
\label{eq:reg}
Y&=&f^*(X) + \varepsilon\,,
\end{eqnarray*}
where
%the observation $X \in \mathcal{X}$ are drawn
%independently from the marginal $P_X$, and
$\varepsilon$ is a noise variable whose distribution can possibly depend on
$X$, but satisfies $\e{\eps|X}=0$.
We assume that the true regression function $f^*$ belongs to the space
$\mathcal{L}_2(P_X)$ of square-integrable functions. Following the kernelization principle,
we implicitly map the data into a reproducing kernel Hilbert space $\cH$ with a kernel
$k$. We denote by  $K_n=\frac{1}{n}(k(X_i,X_j))\in \mathbb{R}^{n\times n}$
the normalized  kernel
matrix and by $\Y=(Y_1,\ldots,Y_n)^\top \in \mathbb{R}^n$  the $n$-vector of response observations. The task is to find coefficients $\alpha$
such that the function defined by the normalized kernel expansion
\begin{eqnarray*}
\label{eq:kernel_exp}
f_\alpha(X)&=&\frac{1}{n}\sum_{i=1} ^n \alpha_{i} k(X_i,X)
\end{eqnarray*}
is an adequate estimator of the true regression function $f^*$.
The closeness of the estimator $f_\alpha$ to the target $f^*$ is measured
via the $\mathcal{L}_2(P_X)$ distance,
\begin{eqnarray*}
\norm{f_\alpha - f^*}^2_{2} &=& \ee{X\sim P_X}{(f_\alpha(X)-f^*(X))^2}\\
&=&
\ee{XY}{(f_\alpha(X)-Y)^2} - \ee{XY}{(f^*(X)-Y)^2}\,,
\end{eqnarray*}
The last equality recalls that this criterion
is the same as the excess generalization error for the squared error loss
$\ell(f,x,y)=(f(x)-y)^2$.

In empirical risk minimization, we use the training data empirical distribution as a proxy for the
generating distribution, and minimize the
 {\em training} squared error. This gives rise to the linear equation
\begin{equation}
\label{eq:kernel_ls}
K_n \alpha =\Y \qquad \text{ with } \alpha \in \mathbb{R}^n\,.
\end{equation}
Assuming $K_n$ invertible, the solution of the above equation is
given by $\alpha = K_n^{-1}\Y$, which yields
a function in $\cH$ interpolating perfectly the training data but having poor
generalization error. It is well-known that to avoid overfitting, some form of
regularization is needed.
There is a considerable variety of possible approaches
(see e.g. \cite{Gyoerfi02} for an overview). Perhaps the most well-known one
is
\begin{equation}
\label{eq:tyk}
\alpha = (K_n + \lambda I)^{-1}\Y,
\end{equation}
known alternatively as kernel ridge regression, least squares support vector
machine, or Tikhonov's regularization. A powerful generalization of
this is to consider
\begin{equation}
\label{eq:linfilter}
\alpha = F_\lambda(K_n) \Y,
\end{equation}
where $F_\lambda: \mbr_+ \rightarrow \mbr_+$ is a fixed function
depending on a parameter $\lambda$. The notation
$F_\lambda(K_n)$ is to be interpreted as $F_\lambda$ applied to
each eigenvalue of $K_n$ in its eigen decomposition. Intuitively,
$F_\lambda$ should be a ``regularized'' version of the inverse function $F(x)=x^{-1}$.
This type of regularization, which we refer to as linear
regularization methods, is directly inspired from the theory of
inverse problems. Popular examples include as particular cases kernel Ridge
Regression, Principal components regression and $L_2$-Boosting.
Their application in a learning context
has been studied extensively \cite{BauPerRos07,Bissantz0701,Cap06,CapDeV07,LogRosetal08}.
Results obtained in this framework will serve as a comparison yardstick in the
sequel.

In this paper, we study conjugate gradient (CG) techniques in combination with early
stopping for the regularization of the kernel based learning problem \eqref{eq:kernel_ls}. The principle of CG techniques is to
restrict the learning problem onto a nested set of data-dependent
subspaces, the so-called Krylov subspaces, defined as
\begin{equation}
\label{eq:Km}
\cK_m(\Y,K_n)= \text{span}\left\{\Y,K_n \Y,\ldots,K_n ^{m-1} \Y  \right\}\,.
\end{equation}
Denote by $\inner{.,.}$ the usual euclidean scalar product on $\mbr^n$ rescaled by the factor $n^{-1}$. We define the $K_n$-norm as
$\norm{\alpha}^2_{K_n}= \inner{ \alpha, K_n \alpha}.$
The CG solution after $m$ iterations is formally defined as
\begin{equation}
\label{eq:crit_cg}
\alpha_m= \text{arg}\min_{\alpha \in \cK_m(\Y,K_n)} \|\Y -K_n \alpha\|_{K_n}\,;
\end{equation}
and the number $m$ of CG iterations is the model parameter.
To simplify notation we define $f_m:=f_{\alpha_m}$. In the learning context considered here, regularization corresponds to early stopping. Conjugate gradients have the appealing property that the optimization
criterion \eqref{eq:crit_cg} can be computed by a simple iterative algorithm
that constructs basis vectors $d_1,\ldots,d_m$ of
$\mathcal{K}_m(\Y,K_n)\,$ by using only {\em forward multiplication}
of vectors by the matrix $K_n$. Algorithm \ref{algo:cg} displays the computation
of the CG kernel coefficients $\alpha_m$ defined by \eqref{eq:crit_cg}.
\begin{algorithm}
\caption{Kernel Conjugate Gradient regression}
\begin{algorithmic}
\STATE{{\bf Input} kernel matrix $K_n$, response vector $\Y$,
maximum number of iterations $m$}
\STATE{{\bf Initialization}: $\alpha_0 ={\bf 0}_n;  r_1 = \Y; d_1 = \Y $}
\FOR{$i=1,\ldots,m$}
\STATE{$d_{i}=d_{i}/\|K_n d_i\|_{K_n}$ (normalization of the basis vector)}
\STATE{$\gamma_i = \inner{\Y, K_n d_i}_{K_n}$ (step size)}
\STATE{$\alpha_{i} = \alpha_{i-1} + \gamma_i d_i$ (update)}
\STATE{$r_{i+1} = r_{i} - K_n \gamma_i d_i$ (residuals)}
\STATE{$d_{i+1} = r_{i+1}  - \inner{K_n d_i,r_{i+1}}_{K_n}/\|r_{i+1}\|^2 _{K_n}$ (new basis vector)}
\ENDFOR
\STATE{{\bf Return:} CG kernel coefficients  $\alpha_m$, CG function $f_m=\sum_{i=1} ^n \alpha_{i,m} k(X_i,\cdot)$}
\end{algorithmic}
\label{algo:cg}
\end{algorithm}

The CG approach is also inspired by the theory of inverse problems, but it is not covered by the framework of linear operators defined in \eqref{eq:linfilter}: As we restrict the learning problem onto the Krylov space $\mathcal{K}_m(\Y,K_n)\,$, the  CG coefficients $\alpha_m$ are of the form $\alpha_m=q_m(K_n) \Y$ with $q_m$ a polynomial of degree $\leq m-1$. However, the polynomial $q_m$ is not fixed but depends on $\Y$ as well, making the CG method nonlinear in the sense that the coefficients $\alpha_m$ depend on $\Y$ in a nonlinear fashion.

We remark that in machine learning, conjugate gradient techniques are often used as fast solvers for operator equations, e.g. to obtain the solution for the regularized equation \eqref{eq:tyk}. We stress that in this paper, we study conjugate gradients  as a {\em regularization approach} for kernel based learning, where the regularity is ensured via early stopping. This approach is not new.  As mentioned in the abstract, the algorithm that we study is closely related to
Kernel Partial Least Squares \cite{Rosipal0101}. The latter method also restricts the learning problem onto the Krylov subspace $\mathcal{K}_m(\Y,K_n)$, but it minimizes the euclidean distance $\|\Y - K_n \alpha\|$ instead of the distance $\|\Y - K_n \alpha\|_{K_n}$ defined above\footnote{This is generalized to a CG-l algorithm ($l \in \mathbb{N}_{\geq 0}$) by replacing the $K_n$-norm in \eqref{eq:crit_cg} with the norm
  defined by $K_n ^l$. Corresponding fast iterative algorithms to compute the
  solution exist for all $l$ (see e.g. \cite{Hanke95}).}. Kernel Partial Least Squares has shown competitive performance in benchmark experiences (see e.g \cite{Rosipal0101,Rosipal0301}). Moreover, a similar conjugate gradient approach for non-definite kernels has been proposed and empirically evaluated by Ong et~al \cite{Ong04}. The focus of the current paper is therefore not to
stress the usefulness of CG methods in practical applications (and we refer to the above mentioned references) but to examine its theoretical
convergence properties. In particular, we establish the existence
of early stopping rules that lead to optimal convergence rates. We summarize our main results in the next session.

\section{Main results}
For the presentation of our convergence results, we require suitable assumptions on the learning problem. We first assume that the
kernel space $\cH$ is separable and that the kernel function is measurable. (This assumption is satisfied for all practical situations that we know of.)
Furthermore, for all results, we make the (relatively standard) assumption that the kernel is {\em bounded}:
\begin{equation}
\label{boundedk}
 \qquad \qquad k(x,x) \leq \kappa \text{ for all } x \in
\cX\,.
\end{equation}
% {\bf and that the kernel is universal??}.  %% NDG: No, this is replaced by the souce condition
We consider -- depending on the result -- one of the following
assumptions on the {\em noise}:
\begin{list}{}{\setlength{\leftmargin}{0cm}}
\item  {\bf (Bounded)} (Bounded $Y$): $|Y| \leq M$ almost surely.
\item {\bf (Bernstein)} (Bernstein condition): $\e{\eps^p | X } \leq (1/2) p! M^p$ almost surely, for all integers $p\geq 2$.
\end{list}
%\vspace{-0.2cm}
The second assumption is weaker than the first. In particular, the
first assumption implies that not only the noise, but also the target function $f^*$ is
bounded in supremum norm, while the second assumption does not put any additional
restriction on the target function.

The {\em regularity} of the target function $f^*$
is measured in terms of a {\em source condition} as follows. The kernel integral operator is given by
\[
K:\mathcal{L}_2(P_X) \rightarrow \mathcal{L}_2(P_X),\,g \mapsto \int k(.,x) g(x) dP(x)\,.
\]
The {\bf source condition} for the parameter $r \geq 0$ is defined by:
\[
\text{\bf SC($r$)}:
f^* = K^r u \qquad \text{ with } \qquad \norm{u} \leq \kappa^{-r}\rho.
\]
It is a known fact (see, e.g., \cite{cuckersmale}) that if $r \geq 1/2$,
then $f^*$ coincides almost surely with a function belonging to $\cH_k$. Therefore, we
refer to $r\geq 1/2$ as the ``inner case'' and to $r<1/2$ as the ``outer case''.

The regularity of the kernel operator $K$ with respect to the marginal
distribution $P_X$ is measured in terms of the {\bf intrinsic dimensionality}
parameter, defined by the condition
\[
\text{\bf ID($s$)}:  \tr(K(K+\lambda I)^{-1}) \leq D^2 (\kappa^{-1} \lambda)^{-s}
\text{ for all } \lambda \in (0,1].
\]

It is known that the best attainable rates of convergence, as a function of the number of
examples $n$, is determined by the parameters $r$ and $s$.
It was shown in \cite{Vito0601} that the minimax learning rate given
these two parameters is lower bounded by
$\mtc{O}(n^{-2r/(2r+s)})$.

We now expose our main results
in different situations. In all the cases considered, the early stopping rule takes the
form of a so-called {\bf discrepancy stopping rule:} For some sequence of thresholds $\Lambda_m$ to be specified (and possibly depending on the data), define the (data-dependent) stopping iteration $\wh{m}$ as the first iteration $m$
for which
\begin{equation}
%\norm{S_n f_{m} - T_n^*\mbf{Y}}_{\cH} < \Lambda_m.
\norm{\left(f_m(X_1),\ldots,f_m(X_n)\right)^\top - \Y}_{K_n} < \Lambda_m\,.
\end{equation}
(Only in the first result below, the threshold $\Lambda_m$ actually depends on the iteration $m$ and on the data.)

\subsection{Inner case without knowledge on intrinsic dimension}
\label{se:innerwos}
The inner case corresponds to $r\geq 1/2$, i.e. the target function $f^*$ lies in $\mathcal{H}$ almost surely.
For some constants $\tau>1$ and $1>\gamma>0$, we consider the discrepancy stopping rule with the threshold sequence
\begin{equation}
\label{eq:disc1}
\Lambda_m = 4\tau \sqrt{\frac{\kappa \log (2\gamma^{-1})}{n}} \paren{ \sqrt{\kappa} \norm{\alpha_{m}}_{K_n} + M \sqrt{\log (2\gamma^{-1})}}.
%\Lambda_m = \tau \paren{ \frac{2}{\sqrt{n}}\log \frac{1}{\gamma} %\norm{\alpha_{m}}_{K_n} + \frac{2M}{\sqrt{n}} \log \frac{1}{\gamma}},
\end{equation}
For technical reasons,
we consider a slight variation of the rule in that
we stop at step $\wh{m}-1$ instead of $\wh{m}$
if $q_{\wh{m}}(0) \geq 4 \kappa \sqrt{\log (2\gamma^{-1})/n}$,
where $q_m$ is the iteration polynomial such that $\alpha_m = q_m(K_n) \Y$. Denote
$\wt{m}$ the resulting stopping step.
We obtain the following result.
\begin{theorem}\label{thm:inner1}
Suppose that $Y$ is bounded  {\bf (Bounded)}, and that the source condition
{\bf SC}($r$) holds for $r\geq 1/2$.
With probability $1-2\gamma$\,, the estimator $f_{\wt{m}}$
obtained by the (modified) discrepancy stopping rule \eqref{eq:disc1}
satisfies
\[
\norm{f_{\wt{m}}-f^*}^2_2 \leq c(r,\tau)
(M+\rho)^2 \paren{  \frac{\log^2 \gamma^{-1}}{n}}^{\frac{2r}{2r+1}}\,.
\]
\end{theorem}
We present the proof in Section \ref{sec:proof}.

\subsection{Optimal rates in inner case}\label{subsec:opt_inner}

We now introduce a stopping rule yielding order-optimal convergence rates as
a function of the two parameters $r$ and $s$ in the ``inner'' case ($r\geq 1/2$,
which is equivalent to saying that the target function belongs to $\cH$ almost surely).
For some constant $\tau'>3/2$ and $1>\gamma>0$, we consider the discrepancy stopping rule with the fixed threshold
\begin{equation}
\label{eq:disc2}
\Lambda_m \equiv \Lambda = \tau' M \sqrt{\kappa} \paren{\frac{4D}{\sqrt{n}} \log \frac{6}{\gamma}}^{\frac{2r+1}{2r+s}}\,.
\end{equation}
for which we obtain the following:
\begin{theorem}\label{thm:inner2}
Suppose that the noise fulfills the Bernstein assumption {\bf (Bernstein)},
that the source condition {\bf SC}($r$) holds for $r\geq 1/2$,
and that {\bf ID($s$)} holds. With probability $1-3\gamma$\,, the estimator $f_{\wh{m}}$
obtained by the discrepancy stopping rule \eqref{eq:disc2}
satisfies
\[
\norm{f_{\wh{m}} - f^*}_2^2 \leq
c(r,\tau')(M+\rho)^2\paren{\frac{16D^2}{n}\log^2 \frac{6}{\gamma}}^{\frac{2r}{2r+s}}\,.
\]
\end{theorem}
Due to space limitations, the proof is presented in the supplementary material.

\subsection{Optimal rates in outer case, given additional unlabeled data}

We now turn to the ``outer'' case. In this case, we make the additional assumption that
{\em unlabeled} data is available. Assume that we have $\tilde n$ i.i.d. observations $X_1,\ldots,X_{\tilde n}$, out of which only the first $n$ are labeled. We define a new response vector $\tilde \Y=\frac{\tilde n}{n} \left(Y_1,\ldots,Y_n,0,\ldots,0\right)\in \mathbb{R}^{\tilde n}$ and run the CG algorithm \ref{algo:cg} on $X_1,\ldots,X_{\tilde n}$ and $\tilde \Y$. We use the same threshold \eqref{eq:disc2} as in the previous section for the stopping rule, except that the factor
$M$ is replaced by $\max(M,\rho)$.
\begin{theorem}\label{thm:outer1}
Suppose assumptions {\bf (Bounded)},
{\bf SC($r$)} and  {\bf ID($s$)}, with
$r+s \geq \frac{1}{2}$.
Assume unlabeled data is available with
\[
%\lambda_* =
\frac{\wt{n}}{n} \geq \paren{\frac{16D^2}{n} \log^2 \frac{6}{\gamma}}^{-\frac{(1-2r)_+}{2r+s}}.
\]
Then with probability $1-3\gamma$\,, the estimator $f_{\wh{m}}$
obtained by the discrepancy stopping rule defined above satisfies
\[
\norm{f_{\wh{m}} - f^*}_2^2 \leq
c(r,\tau')(M+\rho)^2\paren{\frac{16D^2}{n}\log^2 \frac{6}{\gamma}}^{\frac{2r}{2r+s}}\,.
\]
\end{theorem}
A sketch of the proof can be found in the supplementary material.

\section{Discussion and comparison to other results}
For the inner case -- i.e. $f^* \in \mathcal{H}$ almost surely -- we provide two different consistent stopping criteria. The first one (Section \ref{se:innerwos})
is oblivious to the intrinsic dimension parameter $s$, and the obtained bound corresponds to the ``worst case'' with respect to this parameter (that is, $s=1$). However, an interesting feature of stopping rule \eqref{eq:disc1}
is that the rule itself does not depend on the a priori knowledge of the regularity parameter $r$, while the achieved learning rate does
(and with the optimal dependence in $r$ when $s=1$). Hence,  Theorem \ref{thm:inner1} implies that the obtained rule is automatically {\em adaptive} with respect to the regularity of the target function. This contrasts with the results obtained in \cite{BauPerRos07} for
linear regularization schemes of the form \eqref{eq:linfilter},
(also in the case $s=1$) for which the choice of the
regularization parameter $\lambda$ leading to optimal learning rates
required the knowledge or $r$ beforehand.

%We nore that the proof of theorem \ref{thm:inner1} relies on previous results %obtained by Nemirovskii \cite{nemirovskii86} in the context of inverse %problems with perturbations. We present the details in the following section.

When taking into account also the intrinsic dimensionality parameter $s$,
Theorem \ref{thm:inner2} provides the order-optimal convergence rate in the inner case (up to a log factor). A noticeable difference to Theorem \ref{thm:inner1} however, is that the stopping rule is no longer adaptive, that is, it depends on the {\em a priori} knowledge of parameters $r$ and $s$. We observe that previously obtained results for linear regularization schemes of the form \eqref{eq:tyk} in \cite{CapDeV07} and of the form \eqref{eq:linfilter} in \cite{Cap06}, also rely on the a priori knowledge
of $r$ and $s$ to determine the appropriate regularization  parameter $\lambda$.

%{\em should we mention that other people also have non-adaptive bounds and %include some references?}

The outer case -- when the target function does not lie in the reproducing Kernel Hilbert space $\mathcal{H}$ -- is more challenging and to some
extent less well understood. The fact that additional assumptions are made is not a particular artefact of CG methods,
but also appears in the studies of other regularization techniques.
Here we follow the semi-supervised approach that is proposed in e.g. \cite{Cap06} (to study linear regularization
of the form \eqref{eq:linfilter}) and assume that we have sufficient additional unlabeled data in order to ensure learning rates
that are optimal as a function of the number of {\em labeled} data.
We remark that other forms of additional requirements can be
found in the recent literature in order to reach optimal rates.
For regularized M-estimation schemes studied in \cite{Steinwart09}, availability of unlabeled data is not required, but a condition is imposed of the form $\norm{f}_{\infty} \leq C \norm{f}_{\cH}^p \norm{f}_2^{1-p}$ for all $f\in \cH$ and some $p \in (0,1]$.
In \cite{MenNee10},
assumptions on the supremum norm of the eigenfunctions of the
kernel integral operator are made (see \cite{Steinwart09} for an in-depth
discussion on this type of assumptions).

Finally, as explained in the introduction, the term 'conjugate gradients' comprises a class of
methods that approximate the solution of linear equations on Krylov
subspaces. In the context of learning, our approach is most closely linked to
Partial Least Squares (PLS) \cite{Wold8401} and its kernel extension
\cite{Rosipal0101}. While PLS has proven to be  successful
in a wide range of applications and is considered one of the standard approaches in chemometrics, there are  only  few studies of its theoretical properties. In \cite{Chun0901,Naik0001}, consistency properties are provided for linear PLS  under the assumption that the target function $f^*$ depends on a finite known number  of orthogonal latent components. These findings were recently extended to the nonlinear case and without the assumption of a latent components model \cite{Blanchard10}, but all results come without optimal rates of convergence. For the slightly different CG approach studied by Ong et al \cite{Ong04}, bounds on the difference between the empirical risks of the CG approximation and of the target function are derived in \cite{Ong05}, but no bounds on the generalization error were derived.
%------------------------------------------------------------------------------------%
\section{Proofs}\label{sec:proof} %
%------------------------------------------------------------------------------------%
%
%\subsection{Setting: learning as an inverse problem}
%
Convergence rates for regularization methods of the type
\eqref{eq:tyk} or \eqref{eq:linfilter} have been studied by casting
kernel learning methods into the framework of {\em inverse problems}
(see \cite{Vito0601}).
We use this framework for the present results as well,
and recapitulate here some important facts.

We first define the {\em empirical evaluation operator} $T_n$ as follows:
\[
T_n: \qquad g \in \cH \mapsto T_ng :=(g(X_1),\ldots,g(X_n))^\top \in \mbr^n
\]
and the {\em empirical integral operator} $T_n^*$ as:
\[
%\label{eq:kernel_emp}
T_n^*:  u=(u_1,\ldots,u_n) \in \mbr^n \mapsto
T_n^*u := \frac{1}{n} \sum_{i=1}^n  u_i k(X_i,\cdot) \in \cH.
\]
Using the reproducing property of the kernel, it can be readily
checked that $T_n$ and $T_n^*$ are adjoint operators, i.e. they satisfy
$\inner{T^*_n u,g}_{\cH} = \inner{u,T_ng}$, for all $u \in \mbr^n, g\in \cH$\,. Furthermore,  $K_n=T_nT_n^*$, and therefore
$\norm{\alpha}_{K_n} = \norm{f_\alpha}_{\cH}$. Based on these facts,
equation \eqref{eq:crit_cg} can be rewritten as
\begin{equation}
\label{eq:crit_cg2}
f_{m}= \text{arg}\min_{f \in \cK_m(T_n^*\Y,S_n)} \| T_n^*Y - S_n f\|_{\cH}\,,
\end{equation}
where $S_n=T_n^*T_n$ is a self-adjoint operator of $\cH$, called empirical covariance
operator. This definition corresponds to that of the ``usual'' conjugate gradient algorithm
formally applied to the so-called normal equation (in $\cH$)
\[
S_n f_\alpha = T_n^* \Y\,,
\]
which is obtained from \eqref{eq:kernel_ls} by left multiplication by $T_n^*$. The advantage of this reformulation is that it can be interpreted as a ``perturbation''
of a {\em population, noiseless} version (of the equation and of the algorithm), wherein
$Y$ is replaced by the target function $f^*$ and the empirical operator $T_n^*,T_n$ are
respectively replaced by their population analogues, the kernel integral operator
\[
T^* : g \in L_2(P_X) \mapsto T^*g := \int k(.,x) g(x) dP_X(x) = \e{k(X,\cdot) g(X)} \in \cH\,,
\]
and the change-of-space operator
\[
T: \qquad g \in \cH \mapsto g \in \cL_2(P_X)\,.
\]
The latter maps a function to itself
but between two Hilbert spaces which differ with respect to their geometry -- the inner
product of $\cH$ being defined by the kernel function $k$, while the
inner product of $\mathcal{L}_2(P_X)$ depends on the data generating
distribution (this operator is well defined:
since the kernel is bounded, all functions in $\cH$
are bounded and therefore square integrable under any distribution $P_X$).

The following results, taken from \cite{BauPerRos07} (Propositions 21 and 22) quantify more precisely that the empirical covariance operator $S_n = T_n^* T_n$ and the empirical integral operator applied to the data, $T_n ^*\Y$, are close to the population covariance operator $S=T^*T$ and to the
kernel integral operator applied to the noiseless target function, $T^*f^*$ respectively.

\begin{proposition}
\label{prop:hoeftype}
Provided that condition \eqref{boundedk} is true, the following holds:
\begin{equation}
\label{eq:hoeffop}
\prob{\norm{S_n-S}_{HS} \leq
\frac{4\kappa}{\sqrt{n}}\sqrt{\log \frac{2}{\gamma}} }
\geq 1-\gamma\,,
\end{equation}
where $\norm{.}_{HS}$ denotes the Hilbert-Schmidt norm. If the representation $f^*=Tf^*_{\cH}$ holds, and under assumption {\bf (Bernstein)}, we have the following:
\begin{equation}
\label{eq:berndata}
\prob{\norm{T_n^* \bY - S f^*_{\cH}} \leq
\frac{4M\sqrt{\kappa}}{\sqrt{n}} \log \frac{2}{\gamma}}
\geq 1-\gamma\,.
\end{equation}
\end{proposition}
We note that  $f^*=Tf^*_{\cH}$ implies that the target function $f^*$
coincides with a function $f^*_{\cH}$ belonging to $\cH$ (remember that
$T$ is just the change-of-space operator). Hence, the second result \eqref{eq:berndata} is
valid for the case  with $r\geq 1/2$, but  it is not true in general
for $r<1/2\,$.

\subsection{Nemirovskii's result on conjugate gradient regularization rates}
We recall a sharp result due to Nemirovskii \cite{nemirovskii86}
establishing convergence rates for conjugate gradient methods in a
deterministic context. We present the result in an abstract context, then
show how, combined with the previous section, it leads  to a proof of Theorem \ref{thm:inner1}. Consider the linear equation
\[
Az^* = b\,,
\]
where $A$ is a bounded linear operator over a Hilbert space
$\cH$\,. Assume that the above equation has a solution and denote
$z^*$ its minimal norm solution; assume further that a self-adjoint
operator $\bar{A}$, and an element
$\bar{b}\in \cH$ are known such that
\begin{equation}
\label{eq:approxabst}
\norm{A-\bar{A}} \leq \delta\,; \qquad \norm{b-\bar{b}} \leq \eps\,,
\end{equation}
(with $\delta$ and $\eps$ known positive numbers).
Consider the CG algorithm based on the noisy operator $\bar{A}$ and
data $\bar{b}$, giving the output at step $m$
\begin{equation}
\label{eq:cgabst}
z_m = \argmin_{z \in \cK_m(\bar{A},\bar{b})} \norm{\bar{A}z - \bar{b}}^2\,.
\end{equation}
The {\em discrepancy principle} stopping rule is defined as
follows. Consider a fixed constant $\tau>1$ and define
\[
\bar{m} = \min \set{ m \geq 0: \norm{\bar{A}z_m - \bar{b}} <
  \tau(\delta\norm{z_m} + \eps) }\,.
%\text{ or } \norm{\bar{A}z_m -
%    \bar{b}} = \min_z  \norm{\bar{A}z - \bar{b}}}\,.
\]
We output the solution obtained at step
$\max(0,\bar{m}-1)$\,. Consider a minor variation of
%\footnote{It is not clear whether this
%variation is really necessary. In Hanke's version of Nemirovski's
%proof, he does not consider this modification. However, he also
%assumes no error on the operator. We report Nemirovski's original
%result for safety.}
of this rule:
\[
\wh{m} =
\begin{cases}
\bar{m} & \text{ if } q_{\bar{m}}(0) < \eta \delta^{-1}\\
% \text{ and } \norm{\bar{A}z_{\bar{m}-1} -
%    \bar{b}} > \tau' (\delta \norm{z_{\bar{m}}} + \eps)\\
\max(0,\bar{m}-1) &  \text{ otherwise,}
\end{cases}
\]
where $q_{\bar{m}}$ is the degree $m-1$ polynomial such that
$z_{\bar{m}} = q_{\bar{m}}(\bar{A}) \bar{b}$\,, and
%$\tau,\eta$ are
$\eta$ is an arbitrary positive constant
such that
%real numbers such that $\tau'>1$ and
$\eta<1/\tau$\,. Nemirovskii established the following theorem:
\begin{theorem}\label{thm:nemi}
Assume that (a) $\displaystyle \max(\norm{A},\norm{\bar{A}})\leq L;$ and that (b) $\displaystyle z^* = A^\mu u^*$ with $\displaystyle \norm{u^*}\leq R$\, for some $\mu>0$. Then for any $\theta \in [0,1]$\,, provided that $\wh{m}<\infty$ it holds
that
\[
\norm{A^\theta \paren{z_{\wh{m}} - z^*}}^2 \leq
c(\mu,\tau,\eta)R^{\frac{2(1-\theta)}{1+\mu}}(\eps + \delta R L^\mu)^{2(\theta+\mu)/(1+\mu)}\,.
\]
\end{theorem}

\subsection{Proof of Theorem \ref{thm:inner1}}

We apply  Nemirovskii's result in our setting
(assuming $r\geq \frac{1}{2}$): By identifying the
approximate operator and data as $\bar{A}=S_n$ and $\bar{b} =
T^*_n\bY$, we see that the CG algorithm considered by
Nemirovskii \eqref{eq:cgabst} is exactly \eqref{eq:crit_cg2},
more precisely with the identification $z_m = f_{m}$.

For the population version, we identify $A=S$, and $z^* = f^*_{\cH}$
(remember that provided $r\geq \frac{1}{2}$ in the source condition,
then there exists $f^*_{\cH} \in \cH$ such that $f^* = T f^*_{\cH}$).

Condition (a) of Nemirovskii's theorem \ref{thm:nemi} is satisfied with $L=\kappa$
by the boundedness of the kernel.
Condition (b) is satisfied with $\mu = r - 1/2 \geq 0$ and
$R=\kappa^{-r}\rho$,
as implied by the source condition {\bf SC($r$)}. Finally, the concentration result
\ref{prop:hoeftype} ensures that the approximation conditions
\eqref{eq:approxabst} are satisfied
with probability $1-2\gamma$\,, more precisely
with $\delta = \frac{4\kappa}{\sqrt{n}}\sqrt{\log \frac{2}{\gamma}} $ and
$\eps = \frac{4M\sqrt{\kappa}}{\sqrt{n}} \log \frac{2}{\gamma}$.
(Here we replaced $\gamma$ in \eqref{eq:hoeffop} and \eqref{eq:berndata} by
$\gamma/2$, so that the two conditions are satisfied simultaneously, by
the union bound). The operator norm is upper bounded by the Hilbert-Schmidt norm, so that the deviation inequality for the operators is actually stronger than what is needed.

We consider the discrepancy principle stopping rule associated to these parameters, the choice $\eta=1/(2\tau)$, and $\theta=\frac{1}{2}$\,, thus obtaining the result,
since
\[
\norm{A^{\frac{1}{2}} \paren{z_{\wh{m}} - z^*}}^2 =
\norm{S^{\frac{1}{2}} \paren{f_{\wh{m}} - f^*_{\cH}}}^2_{\cH} =
\norm{f_{\wh{m}} - f^*_{\cH}}^2_2 .
\]
% \begin{corollary}
% Assume {\bf IR}($\mu,\rho$) holds.
% With probability $1-\alpha$\,, the estimator $f^{(1)}_{\wh{m}}$
% obtained by the (modified) discrepancy principle stopping rule
% satisfies
% \[
% \norm{f^{(1)}_{\wh{m}}-f^*}_2 \leq c(\mu,\tau,\tau',\eta)
% (M+\rho) \paren{  \frac{\log \alpha^{-1}}{\sqrt{n}}}^{\frac{2\mu % +1}{2\mu+2}}\,.
% \]
% \end{corollary}

\subsection{Notes on the proof of Theorems \ref{thm:inner2} and \ref{thm:outer1}}
The above proof shows that an application of Nemirovskii's fundamental result for CG regularization of inverse problems under deterministic noise (on the data and the operator) allows us to obtain our
first result. One key ingredient is  the concentration property \ref{prop:hoeftype}
which allows to bound deviations in a quasi-deterministic manner.

To prove the sharper results of Theorems \ref{thm:inner2} and \ref{thm:outer1},  such a direct approach does not work unfortunately, and
a complete rework and extension of the proof is necessary.
The proof of Theorem  \ref{thm:inner2}  is presented in the supplementary material to the paper. In a nutshell,
the concentration result \ref{prop:hoeftype} is too coarse to prove the
optimal rates of convergence taking into
account the intrinsic dimension parameter. Instead of that result, we have
to consider the deviations from the mean in a ``warped'' norm, i.e. of the form
$\norm{(S+\lambda I)^{-\frac{1}{2}}(T_n^* \bY - T^*f^*)}$ for the data, and
$\norm{(S+\lambda I)^{-\frac{1}{2}}(S_n-S)}_{HS}$ for the operator (with
an appropriate choice of $\lambda>0$) respectively.
Deviations of this form were introduced and used in \cite{Cap06,CapDeV07} to
obtain sharp rates in the
framework of Tikhonov's regularization \eqref{eq:tyk}
and of the more general linear regularization
schemes of the form \eqref{eq:linfilter}. Bounds on deviations of this form can be obtained via a Bernstein-type concentration inequality for
Hilbert-space valued random variables.

On the one hand, the results
concerning linear regularization schemes of the form \eqref{eq:linfilter}
do not apply to the nonlinear CG regularization. On the other hand,
Nemirovskii's result does not apply to deviations controlled in the warped norm. Moreover, the ``outer'' case introduces additional technical difficulties. Therefore, the proofs for Theorems \ref{thm:inner2} and \ref{thm:outer1}, while still following the overall fundamental structure and ideas introduced by Nemirovskii, are significantly different in that context. As mentioned above, we present the complete proof of Theorem \ref{thm:inner2} in the supplementary material and a sketch of the proof of Theorem \ref{thm:outer1}.

%we combine these techniques with the semi-supervised scenario also %studied in \cite{Cap06}. We feel that the presentation of all %technical details would be beyond the scope of this contribution, and %we hence omit the proof.
%

% For the sake of clarity, we summarize the different operators and % criteria in Table \ref{tab:def}.
% \begin{table}
% \begin{tabular}{|l|l|l|}
% \hline
% description&in $\mathcal{H}_k$&in $L_2(P_X),$ resp. % $L_2(P_{X,n})\equiv \mbr^n$\\
% \hline
% embedding&\multicolumn{2}{|c|}{$T:\mathcal{H}_k \rightarrow % \mathcal{L}_2,g\mapsto g$}\\
% & \multicolumn{2}{|c|}{$T_n:\mathcal{H}_k \rightarrow % \mathcal{L}_2(P_{X_n})=\mathbb{R}^n,g\mapsto % (g(X_1),\ldots,g(X_n))$}\\
% \hline
% covariance operator&$Sw=T^*T w= \int k(x',\cdot )w(x')dx'$&--\\
% &$S_n w= T_n ^* T_n w= \sum_{i=1} ^n k(X_i,\cdot ) w(X_i)$& --\\
% \hline
% kernel operator& -- &$Kg =T T^* g= T \int k(x',\cdot)g(x')dx'$\\
% & -- &$K_n \alpha= T_n T_n ^* \alpha= \sum_{i=1} ^n \alpha_i % k(X_i,\cdot )$\\
% \hline
% least-squares problem& $S_n w= T_n^* \Y$&$K_n \alpha= \Y$\\
% \hline
% conjugate gradient & $w_m= \min_{w} \|S_n w - T_n^* \Y\|^2$ & % $\alpha_m= \min_{\alpha} \|K_n \alpha - \Y\|_{K_n} ^2$\\
% criterion& s.t. $w \in \mathcal{K}_m(T_n^* \Y,S_n)$& s.t. $\alpha \in % \mathcal{K}_m(\Y, K_n)$\\
% \hline
% relationships & $w_m=T^* _n \alpha_m$& --\\
% &$f_m=T w_m$ &$f_m=T T_n ^* \alpha_m$\\
% \hline
% \end{tabular}
% \caption{Representation of the learning problem in  $\cH$
% and in $\mathcal{L}_2$. }
% \label{tab:def}
% \end{table}

\section{Conclusion}
In this work, we derived early stopping rules for kernel Conjugate Gradient regression that provide optimal
learning rates to the true target function. Depending on the situation that we study, the rates are adaptive with respect to the regularity of the target function in some cases. The proofs of our results rely most
importantly on ideas introduced by Nemirovskii \cite{nemirovskii86} and further developed  by Hanke \cite{Hanke95} for CG methods in the deterministic case, and moreover on ideas inspired by \cite{Cap06,CapDeV07}.

Certainly, in practice, as for a large majority of learning algorithms, cross-validation remains
the standard approach for model selection. The motivation of this
work is however mainly theoretical, and our overall goal is to show that from the learning theoretical point of view,
CG regularization stands on equal footing with
other well-studied regularization methods such as kernel ridge regression  or more general linear regularization methods (which includes between many others $L_2$ boosting). We also note that theoretically well-grounded model selection rules can generally help cross-validation in practice by providing a well-calibrated
parametrization of regularizer functions, or, as is the case here, of
thresholds used in the stopping rule.

One crucial property used in the proofs is that the proposed CG regularization schemes can be conveniently cast in the reproducing kernel Hilbert space $\cH$ as displayed in e.g \eqref{eq:crit_cg2}. This reformulation is not
possible for Kernel Partial Least Squares:  It is also a CG type method, but uses the standard Euclidean norm instead of the $K_n$-norm used here. This point is the main technical justification on  why we focus on \eqref{eq:crit_cg} rather than kernel PLS.
Obtaining optimal convergence rates also valid for Kernel PLS is an important future direction and should build on
the present work.

Another important direction for future efforts is the derivation of stopping rules that do not depend on the confidence parameter $\gamma$. Currently,  this dependence prevents us  to go from convergence in high probability to convergence in expectation, which would be desirable. Perhaps more importantly, it would be of interest to
find a stopping rule that is adaptive to both parameters
$r$ (target function regularity) and $s$ (intrinsic dimension parameter) without their a priori knowledge. We recall that our first
stopping rule is adaptive to $r$ but at the price of being worst-case
in $s$. In the literature on linear regularization methods,
the optimal choice of regularization parameter is also non-adaptive,
be it when considering optimal rates with respect to $r$ only
\cite{BauPerRos07} or to both $r$ and $s$ \cite{Cap06}.
An approach to alleviate this problem is to use a hold-out
sample for model selection; this was studied theoretically in \cite{Cap10}
for linear regularization methods (see also \cite{BlaMas06} for an
account of the properties of hold-out in a general setup). We strongly believe that the hold-out method will yield theoretically founded adaptive model selection for CG as well. However,
hold-out is typically regarded as inelegant in that it requires to throw away part of
the data for estimation. It would be of more interest to study
model selection methods that are based on using the whole data in the estimation phase. The application of Lepskii's method is  a possible step towards this direction.

%\bibliographystyle{plain}
%\bibliography{cg_nips2010}

\begin{thebibliography}{10}

\bibitem{Bat97}
R.~Bathia.
\newblock {\em Matrix Analysis}, volume 169 of {\em Graduate texts in
  mathematics}.
\newblock Springer, 1997.

\bibitem{BauPerRos07}
F.~Bauer, S.~Pereverzev, and L.~Rosasco.
\newblock {On Regularization Algorithms in Learning Theory}.
\newblock {\em Journal of Complexity}, 23:52--72, 2007.

\bibitem{Bissantz0701}
N.~Bissantz, T.~Hohage, A.~Munk, and F.~Ruymgaart.
\newblock {Convergence Rates of General Regularization Methods for Statistical
  Inverse Problems and Applications}.
\newblock {\em SIAM Journal on Numerical Analysis}, 45(6):2610--2636, 2007.

\bibitem{Blanchard10}
G.~Blanchard and N.~Kr\"amer.
\newblock {Kernel Partial Least Squares is Universally Consistent}.
\newblock {\em Proceedings of the 13th International Conference on Artificial
  Intelligence and Statistics, JMLR Workshop \& Conference Proceedings},
  9:57--64, 2010.

\bibitem{BlaMas06}
G.~Blanchard and P.~Massart.
\newblock {Discussion of {V}.~{K}oltchinskii's 2004 {IMS} Medallion Lecture
  paper, "{L}ocal {R}ademacher complexities and oracle inequalities in risk
  minimization"}.
\newblock {\em Annals of Statistics}, 34(6):2664--2671, 2006.

\bibitem{Cap06}
A.~Caponnetto.
\newblock {Optimal Rates for Regularization Operators in Learning Theory}.
\newblock Technical Report CBCL Paper 264/ CSAIL-TR 2006-062, Massachusetts
  Institute of Technology, 2006.

\bibitem{CapDeV07}
A.~Caponnetto and E.~De~Vito.
\newblock {Optimal Rates for Regularized Least-squares Algorithm}.
\newblock {\em Foundations of Computational Mathematics}, 7(3):331--368, 2007.

\bibitem{Cap10}
A.~Caponnetto and Y.~Yao.
\newblock {Cross-validation based Adaptation for Regularization Operators in
  Learning Theory}.
\newblock {\em Analysis and Applications}, 8(2):161--183, 2010.

\bibitem{Chun0901}
H.~Chun and S.~Keles.
\newblock {Sparse Partial Least Squares for Simultaneous Dimension Reduction
  and Variable Selection}.
\newblock {\em Journal of the Royal Statistical Society B}, 72(1):3--25, 2010.

\bibitem{cuckersmale}
F.~Cucker and S.~Smale.
\newblock {On the Mathematical Foundations of Learning}.
\newblock {\em Bulletin of the American Mathematical Society}, 39(1):1--50,
  2002.

\bibitem{Vito0601}
E.~De~Vito, L.~Rosasco, A.~Caponnetto, U.~De~Giovannini, and F.~Odone.
\newblock {Learning from Examples as an Inverse Problem}.
\newblock {\em Journal of Machine Learning Research}, 6(1):883, 2006.

\bibitem{Gyoerfi02}
L.~Gyorfi, M.~Kohler, A.~Krzyzak, and H.~Walk.
\newblock {\em {A Distribution-Free Theory of Nonparametric Regression}}.
\newblock Springer, 2002.

\bibitem{Hanke95}
M.~Hanke.
\newblock {\em {Conjugate Gradient Type Methods for Linear Ill-posed
  Problems}}.
\newblock Pitman Research Notes in Mathematics Series, 327, 1995.

\bibitem{LogRosetal08}
L.~Lo~Gerfo, L.~Rosasco, E.~Odone, F.and De~Vito, and A.~Verri.
\newblock {Spectral Algorithms for Supervised Learning}.
\newblock {\em Neural Computation}, 20:1873--1897, 2008.

\bibitem{MenNee10}
S.~Mendelson and J.~Neeman.
\newblock {Regularization in Kernel Learning}.
\newblock {\em The Annals of Statistics}, 38(1):526--565, 2010.

\bibitem{Naik0001}
P.~Naik and C.L. Tsai.
\newblock {Partial Least Squares Estimator for Single-index Models}.
\newblock {\em Journal of the Royal Statistical Society B}, 62(4):763--771,
  2000.

\bibitem{nemirovskii86}
A.~S. Nemirovskii.
\newblock {The Regularizing Properties of the Adjoint Gradient Method in
  Ill-posed Problems}.
\newblock {\em USSR Computational Mathematics and Mathematical Physics},
  26(2):7--16, 1986.

\bibitem{Ong05}
C.~S. Ong.
\newblock {Kernels: Regularization and Optimization}.
\newblock {\em Doctoral dissertation, Australian National University}, 2005.

\bibitem{Ong04}
C.~S. Ong, X.~Mary, S.~Canu, and A.~J. Smola.
\newblock {Learning with Non-positive Kernels}.
\newblock In {\em Proceedings of the 21st International Conference on Machine
  Learning}, pages 639 -- 646, 2004.

\bibitem{PinSak85}
I.~F. Pinelis and A.~I. Sakhanenko.
\newblock Remarks on inequalities for probabilities of large deviations.
\newblock {\em Theory Probab. Appl.}, 1(30):143--148, 1985.

\bibitem{Rosipal0101}
R.~Rosipal and L.J. Trejo.
\newblock Kernel {P}artial {L}east {S}quares {R}egression in {R}eproducing
  {K}ernel {H}ilbert {S}paces.
\newblock {\em Journal of Machine Learning Research}, 2:97--123, 2001.

\bibitem{Rosipal0301}
R.~Rosipal, L.J. Trejo, and B.~Matthews.
\newblock {Kernel PLS-SVC for Linear and Nonlinear Classification}.
\newblock In {\em {Proceedings of the Twentieth International Conference on
  Machine Learning}}, pages 640--647, {Washington, DC}, 2003.

\bibitem{Steinwart09}
I.~Steinwart, D.~Hush, and C.~Scovel.
\newblock {Optimal Rates for Regularized Least Squares Regression}.
\newblock In {\em Proceedings of the 22nd Annual Conference on Learning
  Theory}, pages 79--93, 2009.

\bibitem{Wold8401}
S.~Wold, H.~Ruhe, H.~Wold, and W.J.~Dunn III.
\newblock {The {C}ollinearity {P}roblem in {L}inear {R}egression. The {P}artial
  {L}east {S}quares ({PLS}) {A}pproach to {G}eneralized {I}nverses}.
\newblock {\em {SIAM Journal of Scientific and Statistical Computations}},
  5:735--743, 1984.

\bibitem{Yur95}
V.~Yurinksi.
\newblock {\em Sums and {G}aussian vectors}, volume 1617 of {\em Lecture notes
  in mathematics}.
\newblock Springer, 1995.

\end{thebibliography}

\appendix

\section{Supplementary Material}
\subsection{Notation}

We follow the notation used in the main part, in particular
the operators $T_n, T_n^*, T, T^*$ defined in Section 4.1, and we recall that
$S_n := T_n^*T_n$; $S=T^*T$; $K_n = T_nT_n^*$; and $K=TT^*$.

We denote by $(\xi_i)_{i \geq 1}$ the possibly finite sequence in $[0,\kappa]$ of nonzero
eigenvalues of $S$ and $K$\,,
and by $(\xi_{j,n})_{1 \leq j \leq n}$ the $n$-sequence of eigenvalues
of $S_n$ and $K_n$ respectively (in each case in decreasing order and with multiplicity). Finally,
 $(F_{u})_{u\geq 0}$ denotes the spectral family of the operator $S_n$\,, i.e.
$F_{u}$ is the orthogonal projector on the subspace of $\cH$ spanned by eigenvectors of
$S_n$ corresponding to eigenvalues strictly less than $u$.

It is useful to consider the spectral integral representation:
If $(e_{i,n})_{1\leq i \leq n}$ denotes the orthogonal eigensystem of $S_n$ associated to
the non-zero eigenvalues $(\lambda_{i,n})_{1 \leq i\leq n}$\,,
for any integrable function $h$ on $[0,\kappa]$, we set
\[
\int_0^{\kappa} h(u) d\norm{F_{u,n}T_n^*\bY}^2 :=
\inner{T_n^*\bY,h(S_n)T_n^* \bY } = \sum_{i=1}^n
h(\lambda_{i,n}) \inner{T_n^* \bY,e_{i,n}}^2\,.
\]

By its definition, the output of the $m$-th iteration of the CG algorithm
can be put under the form $f_m = q_m(S_n) T_n^* \bY\,,$ where $q_m \in \cP_{m-1}$\,, the set of real polynomials of degree less than $m-1$\,.
A crucial role is played by the {\em residual polynomial}
\[
p_m(x) = 1 - xq_m(x) \in \cP_{m}^0\,,
\]
where $\cP_{m}^0$ is the set of real polynomials of degree less than
$m$ and having constant term equal to~1. In particular
$T_n^*\bY - S_nf_m = p_m(S_n)T_n^*\bY$. Furthermore, the
definition of the CG algorithm implies that the sequence
$(p_m)_{m\geq 0}$ are orthogonal polynomials for the scalar
product $[.,.]_{(1)}$, where for $i\geq 0$ we define
\begin{align*}
[p,q]_{(i)} & := \inner{p(S_n)T_n^*\bY,S_n^i q(S_n)T_n^*\bY} = \int_0^{\kappa} p(u)q(u)u^i d\norm{F_{u,n}T_n^*\bY}^2\,.\\
\end{align*}
%= \argmin_{p \in \cP_{m}^0}
%\norm{p(S_n)T_n^*\bY}
This can be shown as follows:  $p_m $ is the orthogonal projection,  of the origin onto the affine
  space $\cP_{m}^0 = 1 + x\cP_{m-1}$ with the scalar product
$[.,.]_{(0)}$,\,, where $x\cP_{m-1}$ denotes (with some
  abuse of notation) the set of polynomials of degree less than $m$ with constant
  coefficient equal to zero. Thus $0=[p_m,xq]_{(0)}=[p_m,q]_{(1)}$ for any $q \in
  \cP_{m-1}$\,. From the theory of orthogonal polynomials, it results
  that for any $m \leq m_{final}:= \#\set{i: 1\leq i \leq n,
    \xi_{i,n}\inner{T^*_n \bY,e_{i,n}} \neq 0}$\,, the polynomial
  $p_m$ has exactly $m$ distinct roots belonging to $[0,\kappa]$\,,
  which we  denote by $(x_{k,m})_{1 \leq k \leq m}$
  (in increasing order). Finally, we  use the notation $c(a,b)$ to denote a function depending
on the stated parameters only, and whose exact value can change from line to line.

\subsection{Preparation of the proof}

We follow the general architecture of Nemirovskii's proof to establish rates.
We recall that since we assume $r\geq 1/2$, the representation $f^*
= Tf^*_\cH$ holds.
The main difference to Nemirovskii's original result
is that (similar to the approach of \cite{Cap06,CapDeV07}) we use
deviation bounds in a ``warped'' norm rather than in the standard norm. More precisely, we consider the following
type of assumptions:
\begin{itemize}
\item[{\bf B1($\lambda$)}]
$\displaystyle \norm{(S+\lambda I)^{-\frac{1}{2}}(T_n^*\bY-S_nf^*_\cH)} \leq \delta(\lambda)$\,,
\item[{\bf B2($\lambda$)}]
$\displaystyle \norm{(S+\lambda I)(S_n
  +\lambda I)^{-1}} \leq \Upsilon^2$\,, with $\Upsilon \geq 1$\\
(this implies in particular
$\displaystyle \norm{(S+\lambda I)^{\frac{1}{2}}(S_n
  +\lambda I)^{-\frac{1}{2}}} \leq \Upsilon$ via \eqref{eq:multpert} below)\,,
\item[{\bf B3}]
$\displaystyle \norm{S-S_n} \leq \kappa\Delta$\,.
\end{itemize}

In the rest of this section we set $\mu=r-1/2$.
Under the source condition assumption {\bf SC($r$)}, for
$r\geq\frac{1}{2}$ the
representation $f^* = K^r u$ can be rewritten
\[
f^* = (TT^*)^ru = T (T^*T)^{r-\frac{1}{2}} (T^*T)^{-\frac{1}{2}} T^*u
= T S^\mu (T^*T)^{-\frac{1}{2}} T^*u,
\]
by identification we therefore have the source condition for $f_{\cH}$
given by $f_{\cH} = S^\mu w$ with
$w =(T^*T)^{-\frac{1}{2}} T^*u $, and $\norm{w}_{\cH} \leq \norm{u}$, since $(T^*T)^{-\frac{1}{2}} T^*$ is a restricted isometry from $L_2(P_X)$ into $\cH$.

We define the shortcut notation
\begin{equation}
\label{eq:defZ}
Z_\mu(\lambda) =
\begin{cases}
\lambda^\mu & \text{ for } \mu \leq 1\,, \\
\kappa^\mu \Delta & \text{ for } \mu > 1.\\
\end{cases}
\end{equation}

We start with preliminary technical lemmas, before turning to the proof of Theorem 2.2.

\begin{lemma}
\label{le:lemma1}
For any $\lambda>0$\,, if assumptions {\bf SC($r$)}, {\bf
  B1($\lambda$)}, {\bf B2($\lambda$)} and {\bf B3}
hold, then for any iteration step $1\leq m\leq m_{final}$
\begin{align}
\norm{T_n^*(T_n f_m -\bY)}  \leq & c(\mu) \Upsilon^2 \paren{ \abs{p'_m(0)}^{-(\mu+1)}
  + Z_\mu(\lambda)\abs{p'_m(0)}^{-1}}
\kappa^{-\mu-\frac{1}{2}}\rho \nonumber \\
& + \paren{\abs{p'_m(0)}^{-\frac{1}{2}}+\lambda^{\frac{1}{2}}}\Upsilon \delta(\lambda)\,.
\label{eq:lemma1}
\end{align}
\end{lemma}
\begin{proof}
%Recall first that
%\[
%\norm{T_n^*(T_n g_m -\bY)}^2 = \norm{p_m(S_n)T_n^*\bY}^2 = [p_m,p_m]_{(0)} = \int p_m(u)^2
%d\norm{F_{u,n}T_n \bY}^2\,.
%\]
Recall that $(x_{k,m})_{1 \leq k \leq m}$ denote
the $m$ roots of the polynomial $p_m$\,; define further the function $\varphi_m$ on
the interval $[0,x_{1,m}]$ as
\[
\varphi_m(x) = p_m(x) \paren{\frac{x_{1,m}}{x_{1,m}-x}}^{\frac{1}{2}}
\]
Following the idea introduced by Nemirovski, it can be shown that
\begin{align*}
\norm{T_n^*(T_n f_m -\bY)} & = \norm{p_m(S_n)T_n^*\bY}\\
& \leq \norm{F_{x_{1,m}} \varphi_m(S_n) T_n^* \bY }\\
& \leq \norm{F_{x_{1,m}} \varphi_m(S_n) S_n f^*_\cH} + \norm{F_{x_{1,m}} \varphi_m(S_n)
(T_n^* \bY - S_n f^*_\cH)}:= (I) + (II).
\end{align*}
Above, the first inequality (lemma 3.7. in \cite{Hanke95}) is the crucial point, and relies
fundamentally on the fact that $(p_m)$ is an orthogonal polynomial
sequence.

We start with controlling the second term:
\begin{eqnarray*}
(II) &=& \norm{F_{x_{1,m}} \varphi_m(S_n) (T_n^* \bY - S_n f^*_\cH)}\\
& = &\norm{F_{x_{1,m}} \varphi_m(S_n) (S+\lambda I)^{\frac{1}{2}}
  (S+\lambda I)^{-\frac{1}{2}}(T_n^* \bY - S_n f^*_\cH)}  \\
& \leq& \norm{F_{x_{1,m}} \varphi_m(S_n) (S_n+\lambda I)^{\frac{1}{2}}}
\Upsilon \delta(\lambda)\\
& \leq &\paren{\sup_{x\in[0, x_{1,m}]} x^{\frac{1}{2}} \varphi_m(x) +
\lambda^{\frac{1}{2}}\sup_{x\in[0, x_{1,m}]} \varphi_m(x)} \Upsilon
\delta(\lambda)\\
& \leq &\paren{\abs{p'_m(0)}^{-\frac{1}{2}} + \lambda^{\frac{1}{2}}}
\Upsilon \delta(\lambda)\,,
\end{eqnarray*}
where the last line used the inequality (see (3.10) in
\cite{Hanke95})
\begin{equation}
\label{eq:boundphi}
\sup_{x \in [0,x_{1,m}]} x^\nu \varphi^2_m(x) \leq \nu^\nu \abs{p'_m(0)}^{-\nu}\,,
\end{equation}
for any $\nu \geq 0$ (using the convention $0^0=1$),  which we applied above for $\nu=0,1$\,.
For the first term, we use assumption {\bf SC($r$)};
first consider the case $\mu > 1$:
\begin{align*}
(I) = \norm{F_{x_{1,m}} \varphi_m(S_n) S_n f^*_\cH} & =
\norm{F_{x_{1,m}} \varphi_m(S_n) S_n S^\mu w}\\
& \leq \paren{\norm{F_{x_{1,m}} \varphi_m(S_n) S_n^{\mu+1}} +
  \norm{F_{x_{1,m}} \varphi_m(S_n) S_n}\norm{S^\mu - S_n^\mu}}\kappa^{-\mu-\frac{1}{2}} \rho\\
& \leq c(\mu) \paren{ \abs{p'_m(0)}^{-(\mu+1)} +
  \kappa^\mu \Delta \abs{p'_m(0)}^{-1}}\kappa^{-\mu-\frac{1}{2}}\rho\,,
\end{align*}
where we applied \eqref{eq:boundphi} with $\nu = 2(\mu+1)$, $\nu=2$ and \eqref{eq:controlpowerop}.
\end{proof}
For the case $\mu \leq 1$, using \eqref{eq:multpert} and arguments  similar to the previous case:
\begin{eqnarray*}
(I) &=& \norm{F_{x_{1,m}} \varphi_m(S_n) S_n f^*_\cH} \\
& =&
\norm{F_{x_{1,m}} \varphi_m(S_n) S_n S^\mu w}\\
& \leq &\norm{F_{x_{1,m}} \varphi_m(S_n) S_n (S_n+\lambda I)^{\mu}}
\norm{(S_n+\lambda I)^{-\mu}(S+\lambda I)^\mu} \norm{(S+\lambda)^{-\mu}S^{\mu}}
\kappa^{-\mu-\frac{1}{2}} \rho\\
& \leq& c(\mu) \Upsilon^2 \paren{ \abs{p'_m(0)}^{-(\mu+1)} +
  \lambda^\mu  \abs{p'_m(0)}^{-1}}\kappa^{-\mu-\frac{1}{2}}\rho\,.
\end{eqnarray*}

\begin{lemma}
\label{le:lemma2}
For any $\lambda>0$\,, if assumptions {\bf SC($r$)}, {\bf
  B1($\lambda$)}, {\bf B2($\lambda$)} and {\bf B3}
hold, then for any iteration step $1 \leq m \leq m_{final}$, for any
$\eps \in (0,x_{1,m})$:
\begin{align*}
\norm{T(f_m-f^*_\cH)} \leq & \Upsilon \Bigg(
 3\paren{1+\lambda\paren{\abs{p'_m(0)}+\eps^{-1}}}\Upsilon \delta(\lambda)
+c(\mu)\Upsilon^2 \paren{\eps^{\frac{1}{2}}+\lambda^{\frac{1}{2}}}\paren{\eps^\mu +
   Z_\mu(\lambda)}\kappa^{-\mu-\frac{1}{2}}\rho \\
& +
\sqrt{2}\paren{1+\frac{\lambda^{\frac{1}{2}}}{\eps^{\frac{1}{2}}}}\eps^{-\frac{1}{2}}\norm{T_n^*(T_nf_m-\bY)} \Bigg)
\end{align*}
If $m=0$, the above inequality is valid for any $\eps>0$.
\end{lemma}
\begin{proof}
Set $\bar{f}_m=q_m(S_n)S_nf^*_\cH$\,. This is the element  in $\mathcal{H}$ that we obtain
by applying  the $m$th-iteration CG polynomial $q_m$ to the {\em noiseless} data.  We have
\begin{eqnarray*}
\norm{T(f_m-f^*_\cH)} & = &\norm{S^{\frac{1}{2}}(f_m-f^*_\cH)}\leq \Upsilon \norm{(S_n+\lambda I)^{\frac{1}{2}}(f_m - f^*_\cH)}\\
&\leq&
\Upsilon\Bigg( \norm{F_{\eps}(S_n+\lambda I)^{\frac{1}{2}}(f_m-\bar{f}_m)}
+ \norm{F_{\eps} (S_n+\lambda I)^{\frac{1}{2}}(\bar{f}_m-f^*_\cH)}\\
&& +
  \norm{F_{\eps}^\perp (S_n+\lambda I)^\frac{1}{2}(f_m-f^*_\cH)}\Bigg)\\
  &:= &\Upsilon ( (I) + (II) + (III) )\,,
\end{eqnarray*}
where we denote $F_\eps^\perp := (I-F_\eps)$.  {\em First summand:}
\begin{align*}
(I) = \norm{F_{\eps}(S_n+\lambda I)^{\frac{1}{2}}(f_m-\bar{f}_m)}
 & = \norm{F_{\eps}(S_n+\lambda I)^{\frac{1}{2}}q_m(S_n)
   (S+\lambda I)^{\frac{1}{2}} (S+\lambda I)^{-\frac{1}{2}} \paren{T_n^*
     \bY - S_n f^*_\cH}}\\
& \leq \Upsilon \norm{F_{\eps}(S_n+\lambda I)^{\frac{1}{2}}q_m(S_n)
   (S_n+\lambda I)^{\frac{1}{2}}} \delta(\lambda)\\
& \leq \Upsilon \delta(\lambda) \paren{\sup_{x \in [0,\eps]} x q_m(x)
  + \lambda \sup_{x \in [0,\eps]} q_m(x)}\\
& \leq \Upsilon \delta(\lambda) \paren{1 + \lambda \abs{p'_m(0)}}\,.
\end{align*}
The last inequality is obtained by the following argument: if $m\geq 1$,
since $\eps\leq x_{1,m}$, $p_m$ is convex in $[0,\eps]$\,, we have
\[
q_m(x) = \frac{1-p_m(x)}{x} \leq \abs{p'_m(0)} \qquad \text{ for } x\in[0,\eps]\,;
\]
and also $
xq_m(x) = 1-p_m(x) \leq 1$  for $x\in[0,\eps]\,.$  If $m=0$, we have $p_0 \equiv 1$ and $q_m\equiv 0$, so that the above is also trivially
satisfied for any $x$.

{\em Second summand:} first subcase, $\mu > 1$, using \eqref{eq:controlpowerop},
and the fact that $\abs{p_m}(x) \leq 1$ for $x\in[0, \eps]$:
\begin{eqnarray*}
(II)&=& \norm{F_{\eps} (S_n+\lambda I)^{\frac{1}{2}}(\bar{f}_m-f^*_\cH)}\\
& = & \norm{F_{\eps} (S_n+\lambda I)^{\frac{1}{2}}p_m(S_n)S^\mu w} \\
& \leq &\Bigg( \norm{F_{\eps} (S_n+\lambda I)^{\frac{1}{2}}p_m(S_n)S_n^{\mu}} + \norm{F_\eps(S_n+\lambda I)^{\frac{1}{2}}p_m(S_n)}c(\mu)
  \kappa^{\mu}\Delta \Bigg) \kappa^{-\mu-\frac{1}{2}}\rho\\
& \leq& \paren{\eps^{\mu+\frac{1}{2}} + \lambda^{\frac{1}{2}} \eps^\mu
  + c(\mu)
  \kappa^{\mu} \paren{\eps^\frac{1}{2} +
    \lambda^{\frac{1}{2}}} \Delta} \kappa^{-\mu-\frac{1}{2}} \rho\\
&\leq& c(\mu)\paren{\eps^\frac{1}{2} +
    \lambda^{\frac{1}{2}}}\paren{\eps^{\mu} + \kappa^\mu \Delta }\kappa^{-\mu-\frac{1}{2}}\rho\,.
\end{eqnarray*}
Bounding the second summand: second subcase, $\mu \leq 1$:
\begin{align*}
(II)
 =  \norm{F_{\eps} (S_n+\lambda I)^{\frac{1}{2}}p_m(S_n)S^\mu w}
& \leq \norm{F_{\eps} (S_n+\lambda I)^{\mu+\frac{1}{2}}p_m(S_n)}\Upsilon^2
\kappa^{-\mu-\frac{1}{2}}\rho\\
& \leq c(\mu) (\eps + \lambda)^{\mu+\frac{1}{2}}
\Upsilon^2 \kappa^{-\mu-\frac{1}{2}} \rho\,.
\end{align*}
{\em Third summand:}
\begin{align*}
(III) &= \norm{F_{\eps}^\perp (S_n+\lambda I)^{\frac{1}{2}}(f_m-f^*_\cH)}\leq \norm{F_{\eps}^\perp S_n^{\frac{1}{2}}(f_m-f^*_\cH)} +
\lambda^{\frac{1}{2}} \norm{F_{\eps}^\perp (f_m-f^*_\cH)}\\
& \leq \paren{\frac{(\eps + \lambda)^{\frac{1}{2}}}{\eps^{\frac{1}{2}}} +
  \lambda^{\frac{1}{2}} \frac{(\eps+\lambda)^{\frac{1}{2}}}{\eps}}
\norm{ F_{\eps}^\perp (S_n+\lambda I)^{-\frac{1}{2}}S_n(f_m-f^*_\cH)}\\
&
\leq \paren{1+\frac{\lambda^{\frac{1}{2}}}{\eps^{\frac{1}{2}}}} \paren{1+\frac{\lambda}{\eps}}^{\frac{1}{2}}
\norm{  F_{\eps}^\perp(S_n+\lambda I)^{-\frac{1}{2}}S_n(f_m-f^*_\cH)}\\
& \leq  \paren{1+\frac{\lambda^{\frac{1}{2}}}{\eps^{\frac{1}{2}}}} \paren{1+\frac{\lambda}{\eps}}^{\frac{1}{2}}
\paren{
\norm{F_{\eps}^\perp (S_n+\lambda I)^{-\frac{1}{2}}T_n^*(T_nf_m -\bY)} +
\norm{(S_n+\lambda I)^{-\frac{1}{2}}(T_n^*\bY - S_n f^*_\cH)}}\\
& \leq \paren{1+\frac{\lambda^{\frac{1}{2}}}{\eps^{\frac{1}{2}}}}
\eps^{-\frac{1}{2}}\norm{T_n^*(T_nf_m -\bY)} +
\sqrt{2} \Upsilon \paren{1+\frac{\lambda}{\eps}} \norm{(S+\lambda I)^{-\frac{1}{2}}(T_n^*\bY - S_n f^*_\cH)}\\
&\leq \paren{1+\frac{\lambda^{\frac{1}{2}}}{\eps^{\frac{1}{2}}}}
\eps^{-\frac{1}{2}} \norm{T_n^*(T_nf_m -\bY)} +
\sqrt{2} \Upsilon\paren{1+\frac{\lambda}{\eps}} \delta(\lambda)\,.
\end{align*}
\end{proof}
We now consider the sequence of polynomials that are orthogonal
with respect to the scalar product $[.,.]_{(2)}$, which we denote by
$p_m^{(2)}$, and its roots by $x_m^{(2)}$.

\begin{lemma}
\label{le:relpol}
For any $\lambda>0$\,, if assumptions {\bf SC($r$)}, {\bf
  B1($\lambda$)}, {\bf B2($\lambda$)} and {\bf B3}
hold, then for any iteration step $1 \leq m \leq m_{final}$, for any
$\eps \in (0,x_{1,m-1})$:
\begin{align}
\brac{p_{m-1},p_{m-1}}_{(0)}^{\frac{1}{2}} &
= \norm{p_{m-1}(S_n)T_n^*\bY} \nonumber \\
& \leq \Upsilon (\eps+\lambda)^\frac{1}{2}
\delta(\lambda) + c(\mu)\Upsilon^2\eps\paren{\eps^\mu +
Z_\mu(\lambda)  }\kappa^{-\mu-\frac{1}{2}}\rho +
\eps^{-\frac{1}{2}}\brac{p_{m-1}^{(2)},p_{m-1}^{(2)}}^{\frac{1}{2}}_{(1)}\,.
\label{eq:le3}
\end{align}
\end{lemma}
\begin{proof}
By the optimality property defining our CG algorithm,
\begin{align*}
\norm{p_{m-1}(S_n)T_n^*\bY} &
\leq \norm{p^{(2)}_{m-1}(S_n)T_n^*\bY}\leq \norm{F_{\eps}p^{(2)}_{m-1}(S_n)T_n^*\bY } + \norm{F_\eps^\perp
p^{(2)}_{m-1}(S_n)T_n^*\bY}\\
& \leq \norm{F_\eps T_n^* \bY} + \eps^{-\frac{1}{2}}
\norm{p^{(2)}_{m-1}(S_n)S_n^{\frac{1}{2}}T_n^*\bY} = \norm{F_\eps T_n^* \bY} +\eps^{-\frac{1}{2}}\brac{p_{m-1}^{(2)},p_{m-1}^{(2)}}^{\frac{1}{2}}_{(1)}
\end{align*}
For the last inequality, we have used the fact that
$|p^{(2)}_{m-1}|(x) \leq 1$ for $x \in [0,x_{m-1}^{(2)}]$\,, along
with the assumption $0 < \eps < x_{1,m-1} \leq
x_{1,m-1}^{(2)}$\,; the latter inequality is due to interlacing
properties of the roots of orthogonal polynomials for $[.,.]_{(i)}$
and $[.,.]_{(i+1)}$\, (see \cite{Hanke95}, Cor 2.7).
We now bound
\begin{align*}
\norm{F_\eps T_n^* \bY } & \leq \norm{F_\eps (T_n^* \bY - S_n
  f^*_\cH)} + \norm{F_\eps S_n S^\mu w} \\
& \leq \norm{F_\eps (S_n+\lambda I)^\frac{1}{2}}
\norm{(S_n+\lambda I)^{-\frac{1}{2}}(T_n^*\bY - S_n f^*_\cH)}
+ \norm{F_\eps S_n S^\mu w} \\
& \leq \Upsilon(\eps+\lambda)^{\frac{1}{2}}\delta(\lambda) + \norm{F_\eps S_n S^\mu w}\,;
\end{align*}
for the second term, we divide as usual into two cases: for $\mu > 1$:
\[
\norm{F_\eps S_n S^\mu w} \leq \norm{F_\eps S_n^{\mu+1}w} + \norm{F_\eps S_n(S_n^\mu - S^\mu)w}
\leq \eps c(\mu) \paren{ \eps^\mu + \kappa^{\mu} \Delta} \kappa^{-\mu-\frac{1}{2}}\rho\,,
\]
and for $\mu \leq 1$:
\[
\norm{F_\eps S_n S^\mu w} \leq \norm{F_{\eps}S_n(S_n+\lambda I)^\mu} \Upsilon^2 \kappa^{-\mu-\frac{1}{2}}\rho
\leq \eps(\eps^\mu + \lambda^\mu) \Upsilon^2 \kappa^{-\mu-\frac{1}{2}}\rho\,.
\]
\end{proof}

\subsection{Proof of Theorem 2.2}

We fix
\begin{equation}
\label{eq:lambdastar}
\lambda_* = \paren{\left(4D/\sqrt{n}\right) \log \left(6/\gamma\right)}^{\frac{2}{2\mu+s+1}} \kappa\,.
\end{equation}
and assume $n$ is big enough to ensure $\lambda_* \leq \kappa$\,. Furthermore we denote
$\blambda_* = \kappa^{-1} \lambda_*$ (this normalization was introduced in \cite{Cap06}).

We rewrite equivalently the discrepancy stopping rule as follows: for some fixed $\tau>0$\,,
\begin{equation}
\label{eq:apriorisr}
\wh{m} := \min\set{0 \geq m : \norm{T_n^*(T_n f_m - \bY)} \leq
  (2+\tau)\lambda_*^{\frac{1}{2}} \delta(\lambda_*)}\,,
\end{equation}
where
\begin{equation}
\label{eq:deltastar}
\delta(\lambda_*):= \frac{3}{4} M \blambda_*^{\mu+\frac{1}{2}}\,.
\end{equation}
(Observe that the above
$\tau>0$ is deduced from the constant $\tau'>3/2$ considered in the main part of the paper
via $\tau = \frac{4}{3}(\tau - \frac{3}{2})$.)

% \begin{theorem}
% Assume {\bf (BN)}, {\bf IR($\mu,\rho$)} and {\bf
%   ID($s$)} are satisfied. Then the stopping rule defined by \eqref{eq:apriorisr} is such
% that, with probability larger than $1-3\alpha$\,, the following holds:
% \[
% \norm{f_{\wh{m}} - f^*}_2 \leq
% c(\mu,\tau)(M+\rho)\paren{\frac{4D}{\sqrt{n}}\log \frac{6}{\alpha}}^{\frac{2\mu+1}{2\mu+s+1}}\,.
% \]
% \end{theorem}
%{\bf Note:} the condition $2\mu+s-1\geq 0$ is the same as used by
%Caponnetto to obtain the above rates for linear regularization
%methods, barring the use of additional unlabeled data.

We first check {\bf B1($\lambda_*$)}, {\bf
  B2($\lambda_*$)} and {\bf B3} are satisfied simultaneously with
large probability, using for this concentration
results which are recalled in Section \ref{se:technical}. Concerning {\bf B1($\lambda_*$)}\,,
inequality \eqref{eq:bybn} ensures that with probability
$1-\gamma$\,, we have
\begin{align}
\norm{(S+\lambda_* I)^{-\frac{1}{2}}(T_n^* \bY - S_n f^*_{\cH})} & \leq
 2M\paren{\sqrt{\frac{\cN(\lambda_*)}{n}} + \frac{2\sqrt{\kappa}}{\sqrt{\lambda_*}n}}
\log \frac{6}{\gamma} \nonumber \\
& \leq \frac{2M}{\sqrt{n}} D
\blambda_*^{-\frac{s}{2}} \paren{1 +
  \frac{1}{2D^2}\paren{\frac{4D}{\sqrt{n}} \log \frac{6}{\gamma}}
\blambda_*^{\frac{s-1}{2}}} \log
\frac{6}{\gamma} \nonumber \\
& \leq \frac{M}{2}
\blambda_*^{\mu+\frac{1}{2}} \paren{1+\frac{1}{2D^2}\blambda_*^{\mu+s}}
\nonumber \\
&  \leq \frac{3}{4} M \blambda_*^{\mu+\frac{1}{2}}= \delta(\lambda_*)\,,
\label{eq:defdelta1}
\end{align}
where we have used {\bf SC($r$)}, \eqref{eq:lambdastar} and the
assumptions $D\geq 1$ and $\blambda_* \leq 1$\,.
%Moreover, we assume
%\begin{equation}
%\label{eq:condlambda1}
%\lambda \geq \paren{\frac{2\kappa^2}{Dn}}^{\frac{1}{1-s}}\,,
%\end{equation}
%which ensures essentially that the second summand in
%\eqref{eq:defdelta1} smaller than the first; more precisely, using
%{\bf ID($s$)} we have that {\bf B1($\lambda$)} is satisfied with
%probability $1-\alpha$\,, with
%\[
%\delta(\lambda) := 4M \paren{\frac{D}{\lambda^s n}}^{\frac{1}{2}} \log \frac{6}{\alpha}\,.
%\]
We now turn to {\bf B2($\lambda_*$)}\,. Inequality \eqref{eq:reloperror}
along with a repetition of the above reasoning yields that with
probability $1-\gamma$:
\[
\norm{(S+\lambda_* I)^{-\frac{1}{2}}(S_n-S)}_{HS} \leq
\frac{\sqrt{\kappa}}{M} \delta(\lambda_*)\,,
\]
so that
\[
\norm{(S+\lambda_*I)^{-\frac{1}{2}}(S_n-S)(S+\lambda_* I)^{-\frac{1}{2}}} \leq
\frac{\sqrt{\kappa}}{M} \lambda_*^{-\frac{1}{2}} \delta(\lambda_*)\,.
\]
Observe that
\begin{equation}
\frac{\sqrt{\kappa}}{M} \lambda_*^{-\frac{1}{2}} \delta(\lambda_*) =
\frac{3}{4} \blambda_*^\mu \leq \frac{3}{4}\,,
\end{equation}
so that with Lemma \ref{prop:relopmult}, we obtain that {\bf B2($\lambda_*$)} is satisfied
with $\Upsilon:=2$\, (with probability $1-\gamma$). Finally,
equation (11) in the main paper implies that {\bf (B3)} is also satified with
probability $1-\gamma$, with
\begin{equation}
\label{eq:defDelta}
\Delta := \frac{2}{\sqrt{n}}\log \frac{1}{\gamma}\,.
\end{equation}
To conclude, by the union bound, the event that {\bf B1($\lambda_*$)}, {\bf
  B2($\lambda_*$)} and {\bf B3} satisfied simultaneously has probability larger
than $1-3\gamma$\,, and we assume for the rest of the proof that we
are on this event.

We will assume $\wh{m}\geq 1$\, for the remainder of the proof and postpone
to the end the (simpler) case $\wh{m}=0$.

{\bf First step:} upper bound on $\abs{p'_{\wh{m}-1}(0)}$\,. By definition of the stopping rule we
have $\norm{T_n^*(T_n f_{\wh{m}-1} -\bY)} > (2 + \tau)
\lambda_*^\frac{1}{2} \delta(\lambda_*)$\,. Now applying this together
with the upper bound of Lemma \ref{le:lemma1} we get
\begin{align*}
\tau \lambda_*^\frac{1}{2} \delta(\lambda_*)
& \leq  c(\mu) \paren{\abs{p'_{\wh{m}-1}(0)}^{-(\mu+1)}
  + Z_\mu(\lambda_*)\abs{p'_{\wh{m}-1}(0)}^{-1}}
\kappa^{-\mu-\frac{1}{2}} \rho + 2 \abs{p'_{\wh{m}-1}(0)}^{-\frac{1}{2}} \delta(\lambda_*)\,\\
& \leq 3 \max \paren{2 \abs{p'_{\wh{m}-1}(0)}^{-\frac{1}{2}}
  \delta(\lambda_*), c(\mu) \rho \kappa^{-\mu-\frac{1}{2}}
    \abs{p'_{\wh{m}-1}(0)}^{-(\mu+1)}, c(\mu) \rho \kappa^{-\mu-\frac{1}{2}}
    Z_\mu(\lambda_*)\abs{p'_{\wh{m}-1}(0)}^{-1} }\,.
\end{align*}
We examine in succession the possibility that the maximum in the above
expression is attained for each of the terms which comprise it. If the
first term attains the maximum, this implies $|p'_{\wh{m}-1}(0)| \leq (9/\tau^2) \lambda_*^{-1}\,.$
If the second term attains the maximum, this entails
\[
c(\mu) \rho \kappa^{-\mu-\frac{1}{2}}
    \abs{p'_{\wh{m}-1}(0)}^{-(\mu+1)} \geq \tau \lambda_*^\frac{1}{2} \delta(\lambda_*)\,,
\]
which using \eqref{eq:deltastar} yields:
\[
\abs{p'_{\wh{m}-1}(0)} \leq
c(\mu,\tau) \paren{\frac{\rho}{M}}^\frac{1}{\mu+1} \lambda_*^{-1}\,.
\]
Finally, if the third term attains the maximum, we have
\[
c(\mu) \rho Z_\mu(\lambda_*) \kappa^{-\mu-\frac{1}{2}}
    \abs{p'_{\wh{m}-1}(0)}^{-1} \geq \tau \lambda_*^\frac{1}{2} \delta(\lambda_*)\,,
\]
which using \eqref{eq:deltastar} yields:
\[
\abs{p'_{\wh{m}-1}(0)} \leq
c(\mu,\tau) \frac{\rho}{M} \lambda_*^{-\mu-1}
Z_\mu(\lambda_*)\,.
\]
We now establish the inequality
\begin{equation}
\label{eq:boundZ}
Z_\mu(\lambda_*) \lambda_*^{-\mu} \leq 1\,.
\end{equation}
The inequality is trivial if $\mu \leq 1$ given the definition of $Z_\mu(\lambda_*)$ in \eqref{eq:defZ}.
If $\mu >1$ holds, from the definition \eqref{eq:defDelta}, it holds that $\Delta \leq
\frac{1}{2} \blambda_*^{\frac{2\mu+s+1}{2}}$\,, hence
\[
Z_\mu(\lambda_*) \lambda_*^{-\mu} =
\Delta \blambda_*^{-\mu}
\leq \frac{1}{2} \blambda_*^{\frac{s+1}{2}}
\leq \frac{1}{2}\,.
\]
%since the exponent
%\[
%\frac{1}{2}\min(\mu,1)(2\mu+s+1)-\mu =
%%\begin{cases}
%\frac{\mu}{2} \paren{2\mu+s-1} & \text{ if } \mu\leq 1\\
%\frac{s+1}{2} & \text{ if } \mu \geq 1\,,
%\end{cases}
%\]
%is nonnegative in all cases, because we assumed $2\mu +s -1 \geq 0$\,.
Gathering all three cases, we obtain that it always holds that
\begin{equation}
\label{eq:boundppmm1}
\abs{p'_{\wh{m}-1}(0)} \leq c(\mu,\tau) \max\paren{\frac{\rho}{M},1} \lambda_*^{-1}\,.
\end{equation}

{\bf Second step:} upper bound on $\abs{p'_{\wh{m}}(0)}$\,. For this we use the result of the first step and relate $\abs{p'_{\wh{m}-1}(0)}$ to $\abs{p'_{\wh{m}}(0)}$\,.
It is a property of orthogonal polynomials (see Hanke, Corollary 2.6) that for any $m\geq 1$
\begin{equation}
\label{eq:pcompar}
{{p_{m-1}}'(0)} - {{p_{m}}'(0)} =
\frac{\brac{p_{m-1},p_{m-1}}_{(0)} -
  \brac{p_{m},p_{m}}_{(0)} }{\brac{p^{(2)}_{m-1},p^{(2)}_{m-1}}_{(1)}}
\leq \frac{\brac{p_{m-1},p_{m-1}}_{(0)}}{\brac{p^{(2)}_{m-1},p^{(2)}_{m-1}}_{(1)}}\,.
\end{equation}
To upper bound the above quantity, we apply Lemma \ref{le:relpol}
whithe the choice $\lambda=\lambda_*$ and
\[
\eps = \eps_* := a(\mu,\tau) \min \paren{\frac{M}{\rho},1} \lambda_*\,,
\]
where $0< a(\mu,\tau)\leq 1$ should be chosen small enough in order to satisfy some
constraints to be specified below. The first constraint is the
requirement $\eps_* \in (0,x_{1,m-1})$ in order to apply Lemma
\ref{le:relpol}. For this, it can be seen from \eqref{eq:boundppmm1}
that $a(\mu,\tau)$ can be chosen small enough to ensure
\[
\eps_* \leq \abs{p'_{m-1}(0)}^{-1} \leq x_{1,m-1}\,,
\]
the last inequality is an easy consequence of the fact that $p_{m-1}$
has exactly $(m-1)$ positive real roots and $p_{m-1}(0)=1$\,. We now
turn to upper bound the following quantity appearing on the RHS of \eqref{eq:le3}:
\begin{multline}
\Upsilon(\eps_*+\lambda_*)^\frac{1}{2}
\delta(\lambda_*) + c(\mu)\Upsilon^2\eps_*\paren{\eps_*^\mu +
  Z_\mu(\lambda_*)}\kappa^{-\mu-\frac{1}{2}}\rho\\
\begin{aligned}
& \leq 2(a(\mu,\tau)+1)\lambda_*^{\frac{1}{2}}\delta(\lambda_*)
+ c(\mu)a(\mu,\tau)\min \paren{{\rho},M} \lambda_* \blambda_*^\mu
\kappa^{-\frac{1}{2}}\rho\\
& \leq  (c(\mu)a(\mu,\tau)+2)\lambda_*^{\frac{1}{2}}\delta(\lambda_*)\,,
\end{aligned}
\label{eq:beps}
\end{multline}
where we have used the definition \eqref{eq:deltastar} for
$\delta(\lambda_*)$ and inequality $Z_\mu(\lambda_*) \leq
\lambda_*^{\mu}$, see \eqref{eq:boundZ}\,.
Now, we chose $a(\mu,\tau)$ so that the factor in the last display
satisfies $c(\mu)a(\mu,\tau)\leq \frac{\tau}{2}$\,.
Remember that the definition of
the stopping rule entails
\begin{equation}
\label{eq:resr}
\brac{p_{m-1},p_{m-1}}_{(0)}^{\frac{1}{2}} = \norm{T_n^*(T_n f_{\wh{m}-1} -\bY)} > (2 + \tau)
\lambda_*^\frac{1}{2} \delta(\lambda_*) > (2+\tau)\lambda_*^{\frac{1}{2}}\delta(\lambda_*)\,,
\end{equation}
Now combining
\eqref{eq:le3}, \eqref{eq:resr} and \eqref{eq:beps}, we obtain
\[
\paren{1-\frac{\tau+\frac{1}{2}}{\tau+2}}\brac{p_{m-1},p_{m-1}}_{(0)}^{\frac{1}{2}} \leq
\eps_*^{-\frac{1}{2}} \brac{p_{m-1}^{(2)},p_{m-1}^{(2)}}_{(1)}^{\frac{1}{2}}\,;
\]
using this inequality in relation with \eqref{eq:pcompar} and \eqref{eq:boundppmm1}, we obtain
\begin{align*}
\abs{p_{\wh{m}}'(0)} & \leq \abs{p_{\wh{m}-1}'(0)} + c(\tau)\eps_*^{-1} \leq c(\mu,\tau) \max\paren{\frac{\rho}{M},1} \lambda_*^{-1}\,.
\end{align*}

{\bf Final step.} We apply Lemma \ref{le:lemma2} (with
$\lambda=\lambda_*$ and $\eps=\eps_*$), together with the bound on
$\abs{p_{\wh{m}}'(0)}$ just obtained, and the inequality (by
definition of the stopping rule)
\[
\norm{T_n^*(T_nf_{\wh{m}}-\bY)} \leq (2+\tau) \lambda_*^{\frac{1}{2}} \delta(\lambda_*)\,,
\]
obtaining, using again \eqref{eq:boundZ}:
\begin{align*}
\norm{f_{\wh{m}} - f^*}_2 & = \norm{T(f_{\wh{m}}-f^*_\cH)}\\
& \leq c(\mu,\tau)\paren{ \delta(\lambda_*)
\max\paren{\frac{\rho}{M},1} + \min(\rho,M)\blambda_*^{\mu+\frac{1}{2}}} \leq c(\mu,\tau) (M+\rho)\blambda_*^{\mu+\frac{1}{2}}\,.
\end{align*}
If $\wh{m}=0$, we can apply directly Lemma \ref{le:lemma2} as above
without requiring the two previous steps, since
in this case $p'_0(0)=0$, so that we obtain the same final bound.

% \subsection{Result 2}

% We assume {\bf (BY)}, {\bf IR($\mu,\rho$)} and {\bf
%   ID($s$)} are satisfied. Furthermore, unlabeled data...

\subsection{Sketch of the proof of Theorem 2.3}

For the proof of Theorem 2.3, the condition  {\bf B1($\lambda$)} is replaced by
\begin{itemize}
\item[{\bf B1'($\lambda$)}]
$\displaystyle \norm{(S+\lambda I)^{-\frac{1}{2}}(T_n^*\bY-T^*f^*)} \leq \delta(\lambda)$\,.
\end{itemize}

We check that {\bf B1'($\lambda_*$)}, {\bf B2($\lambda_*$)} and {\bf B3} are satisfied in
the setting of Theorem 2.3.
To check {\bf B1'($\lambda_*$)}, we use \eqref{eq:byby}
instead of \eqref{eq:bybn}. Since the easily checked relation
$T_n^*\bY = T^*_{\tilde n} \wt{\bY}$ holds,
the upper bound obtained here has the same form as
for Theorem 2.2, therefore we can use the same value $\delta(\lambda^*)$ for
condition {\bf B1'($\lambda_*$)} as in the previous section, given by
\eqref{eq:defdelta1}. Notice however that we must now use the condition $\mu+s = r + s - \frac{1}{2} \geq 0$
to ensure that the chain of inequalities leading to \eqref{eq:defdelta1} is valid.

For condition {\bf B2($\lambda_*$)}, we can apply the deviation inequality \eqref{eq:reloperror}
but with $n$ replaced by $\wt{n}$, since we make use of all the unlabeled data. Using the fact that $\frac{n}{\tilde n} \leq {\blambda_*}^{-(1-2r)_+}$
and some elementary algebra leads to {\bf B2($\lambda_*$)} being
satisfied with $\Upsilon:=2$.

Finally condition {\bf B3} is satisfied with $\Delta$ given by
\eqref{eq:defDelta} with $n$ replaced by ${\tilde n}$.

Once these conditions are established, intermediate results similar
in structure to Lemmas \ref{le:lemma1}, \ref{le:lemma2}
and \ref{le:relpol} can be derived, but where {\bf B1($\lambda$)} is replaced by {\bf B1'($\lambda$)}.
The details are omitted here.

\subsection{More technical lemmas}
\label{se:technical}

In this section we collect some technical lemmas which
underpin the main results.  These  are taken from
previous sources and are recalled here for completeness.
The main statistical tool is the following deviation inequality:

%  a Bernstein-type inequality for
% random variables taking values in a Hilbert space:

% \begin{theorem}[ \cite{PinSak85,Yur95} as cited in \cite{Cap06,CapDeV07}]
% \label{th:bernstein_hilbert}
% Let $\xi$ be a random variable on a probability space $(\Omega, \cF, P)$ taking values in a
% real separable Hilbert space $\cH$. Assume that there are constants $H,\sigma$ such that
% \begin{align*}
% \norm{\xi} & \leq \frac{H}{2} \;\; \text{a.s.};\\
% \e{\norm{\xi}^2} & \leq \sigma^2,
% \end{align*}
% then, for all $n \in \mbn$ and $0<\gamma<1$\,,
% \[
% \probb{(\om_1,\ldots,\om_n) \sim P^{\otimes n}}{\norm{\frac{1}{n}\sum_{i=1}^n \xi(\omega_i) -
% \e{\xi}} \leq 2 \paren{\frac{\sigma}{\sqrt{n}} + \frac{H}{n} }\log \frac{2}{\gamma}} \geq 1-\gamma\,.
% \]
% \end{theorem}

% Application of the previous theorem yields the following results in the framework considered here:
%let us first define
%\begin{align*}
%\delta_n(\lambda)
%&= 2M\paren{\sqrt{\frac{\cN(\lambda)}{n}} + \sqrt{\frac{\kappa}{\lambda}}\frac{2}{n}}\\
%\Delta_n(\lambda)
%& =
%\end{align*}
%We then have:
\begin{lemma}
\label{prop:relconc}
Let $\lambda$ be a positive number. Under assumption {\bf (Bounded)}, the following holds:
\begin{equation}
\label{eq:byby}
\prob{\norm{(S+\lambda I)^{-\frac{1}{2}}(T_n^* \bY - T^*f^*)} \leq
2M\paren{\sqrt{\frac{\cN(\lambda)}{n}} + \frac{2\sqrt{\kappa}}{\sqrt{\lambda}n}}
\log \frac{6}{\gamma} }
\geq 1-\gamma\,.
\end{equation}
If the representation $f^*=Tf^*_{\cH}$ holds and under assumption {\bf (Bernstein)},
we have the following:
\begin{equation}
\label{eq:bybn}
\prob{\norm{(S+\lambda I)^{-\frac{1}{2}}(T_n^* \bY - S_n f^*_{\cH})} \leq
2M\paren{\sqrt{\frac{\cN(\lambda)}{n}} + \frac{2\sqrt{\kappa}}{\sqrt{\lambda}n}}
\log \frac{6}{\gamma}}
\geq 1-\gamma\,.
\end{equation}
Finally, the following holds:
\begin{equation}
\label{eq:reloperror}
\prob{\norm{(S+\lambda I)^{-\frac{1}{2}}(S_n-S)}_{HS} \leq
2\sqrt{\kappa} \paren{ \sqrt{\frac{\cN(\lambda)}{ n}} + \frac{2\sqrt{\kappa}}{\sqrt{\lambda} n}}
\log \frac{6}{\gamma} }
\geq 1-\gamma\,,
\end{equation}
where we recall that $\norm{.}_{HS}$ denotes the Hilbert-Schmidt norm.
\end{lemma}
The proof can be found in \cite{CapDeV07}, and is based on a
Bernstein-type inequality for random variables taking values in a Hilbert space,
as established in \cite{PinSak85,Yur95}.
Inequality \eqref{eq:reloperror} can be fruitfully combined with the following:
\begin{lemma}
\label{prop:relopmult}
Assume there exists $\eta>0$ such that the following inequality holds:
\[
\norm{(S+\lambda)^{-\frac{1}{2}}(S_n-S) (S+\lambda)^{-\frac{1}{2}}} < 1 - \eta\,,
\]
then
\[
\norm{(S+\lambda)^{\frac{1}{2}}(S_n + \lambda)^{-\frac{1}{2}}} \leq \frac{1}{\sqrt{\eta}}.
\]
\end{lemma}
\begin{proof}
First we have
\[
\norm{(S+\lambda)^{\frac{1}{2}}(S_n + \lambda)^{-\frac{1}{2}}}
= \norm{(S+\lambda)^{\frac{1}{2}}(S_n + \lambda)^{-1}(S+\lambda)^{\frac{1}{2}}}^{\frac{1}{2}};
\]
then simple algebraic manipulation shows
\[
(S+\lambda)^{\frac{1}{2}}(S_n + \lambda)^{-1}(S+\lambda)^{\frac{1}{2}}
= \paren{ I - (S+\lambda)^{\frac{1}{2}}(S - S_n)^{-1}(S+\lambda)^{\frac{1}{2}}}^{-1}.
\]
Finally, using the inequality $\norm{(I-A)^{-1}} = \norm{\sum_{k\geq
    0} A^k} \leq (1-\norm{A})^{-1}$ for $\norm{A} < 1$  yields the conclusion.
\end{proof}

We make use of the following operator inequalities:
\begin{lemma}
Let $A,B$ be two positive, self-adjoint operators with
$\max(\norm{A},\norm{B}) \leq C$\,. Then for any $r\geq 0$\,,
putting $\zeta=(r-1)_+$\,, the following inequality holds:
\begin{equation}
\label{eq:controlpowerop}
\norm{A^r - B^r} \leq (\zeta+1) C^{\zeta} \norm{A-B}^{r-\zeta}\,.
\end{equation}
%$c_r = 1$ for $r < 1 $ and $r=rC^{r-1}$ for $r\geq 1$\,.
\begin{proof}
Follows from the fact that the power function $x \mapsto x^r$ is
operator monotone for $r\leq 1$ and Lipschitz with constant $rC^{r-1}$
on $[0,C]$ if $r>1$.
\end{proof}
\end{lemma}

\begin{lemma}[\cite{Bat97}, Theorem IX.2.1-2]
Let $A,B$ be to self-adjoint, positive operators. Then for any $s\in[0,1]$:
\begin{equation}
\label{eq:multpert}
\norm{A^sB^s} \leq \norm{AB}^s\,.
\end{equation}
\end{lemma}
Note: this result is stated for positive matrices
in \cite{Bat97}, but it is easy to check that the proof applies
as well to positive operators on a Hilbert space.

\end{document}